\documentclass[11pt, a4paper]{amsart}
\input xy   
\xyoption{all}
\numberwithin{equation}{section}
\swapnumbers

\theoremstyle{plain}
  \begingroup

        \newtheorem{proposition}[equation]{Proposition}

        \newtheorem{remark}[equation]{Remark}
	\newtheorem{definition}[equation]{Definition}

  \endgroup

\theoremstyle{definition}
  \begingroup

  \endgroup

\newcommand{\mr}[1]{\buildrel {#1} \over \longrightarrow}
\newcommand{\ml}[1]{\buildrel {#1} \over \longleftarrow}

\newcommand{\rimply}{\Rightarrow}

\newcommand{\cc}{\mathcal}

\begin{document}

\title{2-filteredness and the point of every Galois topos}

\author{Eduardo J. Dubuc}

\begin{abstract}
A locally connected topos is a \emph{Galois topos} if the Galois
objects generate the topos. We show that the full subcategory of
Galois objects in any connected locally connected topos is an inversely 2-filtered 2-category, and as an
application of the construction of 2-filtered bi-limits of topoi, we
show that every Galois topos has a point.
\end{abstract}

\maketitle


\vspace{2ex}

{\sc introduction.} Galois topoi (definition \ref{galoistopos}) arise in Grothendieck's Galois theory
of locally connected topoi. They are an special kind of atomic
topoi. It is well known that atomic topoi may be pointless \cite{Mk}, however,
in this paper we show that any Galois topos has points. 

We show how
the full subcategory of Galois objects (definition
\ref{galoisobjects}) in any connected locally connected topos $\cc{E}$
has an structure of 2-filtered 2-category (in the
sense of \cite{DS}). Then we show that the assignment, to each Galois
object $A$, of the category $\cc{D}_A$ 
of connected locally constant objects trivialized by $A$ (definition \ref{D_A}), determines a
2-functor into the category of categories. Furthermore, this
2-system becomes a pointed 2-system of pointed sites (considering the
topology in which each single arrow is
a cover). By the results on 2-filtered bi-limits of topoi \cite{DY}, it
follows that, if $\cc{E}$ is a Galois topos, then it is the bi-limit of
this system, and thus, it has a point.   

\vspace{1ex}

{\sc context.} 
Throughout this paper $\cc{S} = Sets$ denotes the topos of
sets. All topoi $\cc{E}$ are assumed to be Grothendieck topoi
(over $\cc{S}$), the structure map will be denoted by 
$\gamma: \cc{E} \to \cc{S}$ in all cases. 

\section{Galois topoi and the 2-filtered 2-category of Galois objects}
\label{cpts} 

We recall now the definition
of Galois object in a topos. The original definition of Galois object given in \cite{G1} was
relative to a surjective point of the topos:
 
\begin{definition} \label{chargalois}
Let $\mathcal{E} \mr{\gamma} \cc{S}$ be a topos furnished with a surjective point, that is, a geometric morphism $\cc{S} \mr{p}  
\cc{E}$ whose inverse image functor reflects isomorphisms. Then, an
object $A$ is a Galois object if:

i) There exists $a \in p^{*} A$
such that the map $Aut(A) \mr{a^{*}} p^{*} A$, defined by \mbox{$a^{*}(h) =
p^{*}(h)(a)$} is a bijection (the same holds then for any other $b
\in p^{*}A$). 

ii)  $A$ is connected and $A \to 1$ is epimorphic.
\end{definition}
Notice that in the context of the classical 
Galois theory (Artin's interpretation) this definition coincides with the definition of normal extension.

It is easy to check that in the presence of a (surjective) point the following unpointed definition is equivalent. 
\begin{definition} \label{galoisobjects}
An object $A$ in a topos $\gamma: \cc{E} \to \cc{S}$ is a Galois
object if:

i) The canonical map  $A \times \gamma^{*} Aut(A) \mr{}  A \times A$, 
described by \mbox{$(a,\, h) \mapsto (a, \, h(a))$,} is an isomorphism.

ii) $A$ is connected and $A
\to 1$ is epimorphic.

In particular, $A$ is a connected locally constant object. It follows:

iii) ``$Z \subset A  \hspace{4ex} \rimply  \hspace{4ex} Z =
\emptyset \hspace{2ex} or \hspace{2ex} Z = A$''.   
\end{definition} 
\begin{proposition} \label{abajo}
Let $A$ be a Galois object in a locally connected topos 
\mbox{$\cc{E} \mr{\gamma}
\cc{S}$,} and let $X$ be any object such that there exists an
epimorphism $A \mr{e} X$. Then, the canonical map $A \times
\gamma^{*}[A,\, X] \,\to\,  A \times X$, 
described by $(a,\, f) \mapsto (a, f(a))$, is an isomorphism.
In particular, $X$ is a connected locally constant object split by
the cover $A \to 1$. 
\end{proposition}
\begin{proof}
Consider the following commutative diagram:
$$
\xymatrix@C=8ex
         {
           A \times \gamma^{*} Aut(A) \ar[r]^(.6){\cong} 
                                      \ar[d]^{A \times \gamma^* e_*}
         & A \times A \ar[d]^{A \times e}
        \\
           A \times \gamma^{*}[A,\, X] \ar[r] 
         & A \times X
         } 
$$
This shows that the map (bottom row) is an epimorphism. To see that it
is also a 
monomorphism, let
$ \xymatrix
          {
             Z  \ar@<1ex>[r]^(.3){s}
             \ar@<-1ex>[r]^(.3){t}  
          &  A \times \gamma^{*}[A,\, X]
          }
$
(with $Z \neq \emptyset$ and connected) be a pair of maps which become equal into $A \times X$.
The maps $s$ and $t$ are given by pairs $(u,\, f)$ and $(v,\, g)$,
with $\xymatrix{Z  \ar@<1ex>[r]^{u}
                   \ar@<-1ex>[r]^{v} & A}$
and  $\xymatrix{A  \ar@<1ex>[r]^{f}
                   \ar@<-1ex>[r]^{g} & X}$,
such that $u = v$ and $f\circ u = g \circ v$. Then,  $Equalizer(f,\,g)
\neq \emptyset$. It follows from iii) in definition
\ref{galoisobjects} that $f = g$. Since connected objects generate the
topos, this finishes the proof. 
\end{proof}

\vspace{1ex}

Given any two Galois objects $A$, $B$, in a connected locally
connected topos $\cc{E}$, any connected component of the product $A
\times B$ is a connected locally constant object. It follows from the existence of Galois closure (see
for example \cite{D2} A.1.4) that 
there is a Galois object $C$ and morphisms $C \to A$, $C \to B$. The
full subcategory $\cc{A}$ of Galois objects fails to be (inversely)
filtered because, clearly, different morphisms 
$\xymatrix{A  \ar@<1ex>[r]^{u}  \ar@<-1ex>[r]_{v}  & B}$ between
Galois objects  cannot be equalized from a Galois object $C \to A$
unless they are already equal. However, we have:
\begin{proposition} \label{2filtered}
The category $\cc{A}$ of Galois objects in a connected locally
connected topos becomes
a (inversely) 2-filtered 2-category (in the sense of \cite{DS}) by adding a formal 2-cell  
$\xymatrix@C=8ex{A  \ar@<1.5ex>[r]^{u}
  \ar@<-1.5ex>[r]_{v}^{\Downarrow \theta_{vu}}  & B}$ between any two
morphisms, satisfying the following equations: 
$$
\theta_{uu} = id, \;\; \theta_{wv} \circ \theta_{vu} = \theta_{wu}, \;\;
 (thus \; \theta_{uv} = \theta^{-1}_{vu}), \;\; \theta_{vu} \, \theta_{sr} = \theta_{sv \, ru}
$$
\end{proposition}
\begin{proof}
the proof is very easy, we let the interested reader look at the
definition of 2-filtered 2-category given in \cite{DS} and verify the assertion.
\end{proof}

After Grothendieck ``Categories Galoisiennes'' of \cite{G1}
and Moerdiejk ``Galois Topos'' of \cite{M2}, we state the following
definition:

\begin{definition} \label{galoistopos}
A Galois Topos is a connected locally connected topos generated by its
Galois objects, or, equivalently, such that any connected object is
covered by a Galois object (notice that we do not require the topos
to be pointed). 
\end{definition}

Since Galois objects are connected locally constant
objects, it follows that Galois topoi are generated by locally
constant objects. On the other hand, the existence of Galois closure (see for
example \cite{D2} 
A.1.4) shows that any such topos is a Galois topos. Thus, \emph{a
  connected  topos is a Galois
topos if and only if it is generated by its connected locally constant
objects.}. It follows:

\begin{proposition} \label{atomic}
Any Galois topos is a connected atomic topos; that is, is a connected locally
connected boolean topos.
\end{proposition}

\section{Galois Topoi as filtered bi-limits of topoi with points.}
 
 Consider now a connected locally connected topos $\cc{E}$, and let
 $\cc{C}$ be  
 a full subcategory of connected generators. Let $A \in \cc{C}$ be a
 Galois object. We denote by $\cc{C}_A \subset \cc{A}$ the full
 subcategory whose objects are the $X \in \cc{C}$ below $A$ (remark that this is not the comma category $(A,\, \cc{C})$). If there is a morphism $A \to B$ between Galois
 objects, clearly $\cc{C}_B \subset \cc{C}_A$, and if $\cc{E}$ is a
 Galois topos, by definition the category $\cc{C}$ is the filtered
 union:
$$
\xymatrix
         {
           & .\,.\,. & \cc{C}_B\;\; \ar@{^{(}->}[r] 
           & \;  \cc{C}_A\;\;   \ar@{^{(}->}[r]
           & \;\;.\,.\,. \;\;\;  \ar@{^{(}->}[r] \ & \;\;\; \cc{C}
         }
$$
The topos $\cc{E}$ is the topos of sheaves for the canonical topology
on $\cc{C}$, and each $\cc{C}_A$ is itself a site with this topology. It follows from the theory of filtered inverse bi-limits of topoi
(\cite{G3} Expose VI) that, if $\cc{E}_A$ is the topos of sheaves on
$\cc{C}_A$, then the topos $\cc{E}$ is an inverse bi-limit of topoi: 
$$\cdots  \hspace{5ex} \cc{E}_B  \ml{} \cc{E}_A  \; \ml{}\;\;\;\;\;
\cdots   \;\;\; \ml{}  \; \cc{E}$$

The representable functor $\cc{C}_A \mr{[A,\, -]} \cc{S}$ is a point
of the site, thus the topoi $\cc{E}_A$ are all pointed topoi. However,
a point for the site $\cc{C}$ is equivalent to a simultaneous choice
of points for each $A$ commuting with all the inclusions  
 $\cc{C}_B \subset \cc{C}_A$. That is, an element of the inverse limit
of sets:
$$\cdots  \hspace{5ex} Points(\cc{C}_B)  \ml{} Points(\cc{C}_A)  \; \ml{}\;\;\;\;\;
\cdots   \;\;\; \ml{}  \; Points(\cc{C})$$
which, a priori, may be empty.

\section{Galois Topoi as  pointed 2-filtered bi-limits of pointed topoi.}
 

We shall consider next a different category associated to any Galois
object. 
\begin{definition} \label{D_A}
Let $\cc{E} \mr{\gamma^*} \cc{S}$ be a connected locally connected
topos, let $\cc{C}$ be the subcategory of connected objects, and let
$A \in \cc{C}$ be any Galois object. The category $\cc{D}_A$ is defined
as the bi-pullback of categories:
$$
\xymatrix
         {
            \cc{D}_A  \ar[r]  \ar[d]
          & \cc{C} \ar[d]^{A \times (-)}
          \\
            \cc{S} \ar[r]^{\gamma^*} 
          & \cc{E}_{/A}
         }
$$ 
The objects and arrows of $\cc{D}_A$ can be described as follows
 (where $\pi_1$ denotes the first projection):
$$Ob:\; triples \; (X,\, S,\, \sigma),\;  X \in \cc{C},\; \emptyset
 \neq S \in
 \cc{S}, \;\; (\pi_1,\, \sigma): A \times \gamma^* S \mr{\cong} A \times  X$$
$$Arr: \; (X, \, S,\, \sigma) \to (Y,\, T,\, \xi): \;
X \mr{f} Y,\; S \mr{\eta} T \; | \;\; \xi (a, \, \eta (s)) = f(\sigma(a,
\,s)).
$$
\end{definition}
Given any  $(X,\, S,\, \sigma) \in  \cc{D}_A$, since $S \neq
\emptyset$ it follows that $[A, \, X] \neq \emptyset$. Thus, $X \in
\cc{C}_A$. Furthermore, since $X$
is connected and locally constant, any map $A \to X$ is an
epimorphism. In fact, for the same reason, any map $X \to Y$ is an
epimorphism. It follows that so it is $A \times
\gamma^*(\eta)$. Since $A$ is connected, it follows that $\eta$ is
also an epimorphism, thus a surjective function of sets.    

\begin{remark} \label{epi}
Given any arrow $(X, \, S,\, \sigma) \mr{f, \, \eta} (Y,\, T,\,
\xi)$, $f$ is an epimorphism and $\eta$ a surjective function.
\end{remark}
\begin{proposition} \label{equivalencia}
The functor $\cc{D}_A \mr{\sim} \cc{C}_A$, $(X,\, S,\, \sigma)
\mapsto X$,  is an equivalence of categories. $\cc{D}_A$ has a site
  structure such that any single arrow  is a cover. The induced morphism $\cc{E}_A \mr{\sim} \cc{P}_A$ is an equivalence (where $\cc{P}_A$ denotes the topos of sheaves on $\cc{D}_A$).
\end{proposition}
\begin{proof}
Just by definition of connected object it immediately follows that
given  $(X, \, S,\, \sigma)$, $(Y,\, T,\, \xi)$, and an
arrow $X \mr{f} Y$, there exists a unique \mbox{$S \mr{\eta} T$} which
determines an arrow $(f, \,\eta)$ in $\cc{D}_A$. Thus, the functor is
full and faithful. That it is essentially surjective follows by
proposition \ref{abajo}. The second assertion is clear (consider
remark \ref{epi}).
\end{proof}

\begin{proposition}  \label{2cells} $ $

a) The functor  
$\cc{D}_A \mr{p^*_A} \cc{S}$, $p^*_A(X,\, S,\, \sigma) = S$,
determines a point of the site. This point is naturally isomorphic to
the representable functor $[A, \, -]$ under the equivalence  $\cc{D}_A \mr{\sim} \cc{C}_A$.

b) A morphism between Galois objects
$A \mr{u} B$ determines a morphism of sites $\cc{D}_B \mr{u^*}
\cc{D}_A$ commuting with the points $p^*_A \circ u^* = p^*_B$.

c) Given any two morphisms between Galois objects 
$\xymatrix@C=3.2ex{A  \ar@<1ex>[r]^{u}  \ar@<-1ex>[r]_{v}  & B}$, there is a
canonical natural transformation 
$\xymatrix@C=8ex{\cc{D}_A  \ar@<1.5ex>[r]^{u^*}
  \ar@<-1.5ex>[r]_{v^*}^{\Downarrow \theta_{vu}}  & \cc{D}_B}$ 
satisfying the equations in definition \ref{2filtered}

d) The following diagram commutes:
$$\hspace{10ex} 
 \xymatrix@C=8ex
          {
             \cc{C}_B  \ar@{^{(}->}[r] 
          &  \cc{C}_B  
          \\
             \cc{D}_A  \ \ar[u]^{\cong} \ar@<1.5ex>[r]^{u^*}
             \ar@<-1.5ex>[r]_{v^*}^{\Downarrow \theta_{vu}}  
          &  \cc{D}_B \ar[u]^{\cong}
          }
$$  
\end{proposition}
\begin{proof} $ $

a) It follows from remark \ref{epi} and the fact that $A$ is
connected. Furthermore, we know there exists an epimorphism $A \to X$. The natural bijection $[A,\, X] \cong S$ follows in the same way as item c) below (recall proposition \ref{abajo}).     

b) follows by the universal property of bi-pullbacks. Given
$A \mr{u} B$, an explicit construction of 
$u^*$ is the following: $(X,\, S,\, \sigma) = u^*(Y,\, T,\, \xi)$,
$X = Y$, $S = T$, and $\sigma$ is the map uniquely determined by the
equation $\sigma(s,\, a) = \xi(s,\, u(a))$. 

c) Consider the description in b) and the following diagram:
$$
\xymatrix
         { 
            u^*(Y,\, T,\, \xi):  \ar[d]^{\theta_{v u}}
          & A \times \gamma^*S  \ar[r]^{\cong} 
                                \ar[d]^{A \times \gamma^*(\eta)}
          & A \times X \ar[d]^{id}
          \\            
            v^*(Y,\, T,\, \xi):
          & A \times \gamma^*S  \ar[r]^{\cong}
          & A \times X
         }
$$ 
By definition of connected object, there exists a unique $S
\mr{\eta} S$ making the square commutative. Define 
$\theta_{v u} = (id_X,\, \eta)$. Clearly, the equations hold by
uniqueness.

Finally, d) is clear by definition of $u^*$, $v^*$ and  $\theta_{v u}$ 
\end{proof}


{\sc Every Galois topos $\cc{E}$ has a point:}

It follows from proposition \ref{2cells} b) and c) that the assignment
of the site $\cc{D}_A$ to a Galois object $A$ determines a 2-filtered
2-system of categories which has a bi-colimit that we denote $\cc{D}$
(see \cite{DS}).
$$
\xymatrix@C=8ex
         {
           & .\,.\,. & \cc{D}_B\;\;   \ar@<1.5ex>[r]^{u}
  \ar@<-1.5ex>[r]_{v}^{\Downarrow \theta_{vu}}
           & \;  \cc{D}_A\;\; \ar@<1.5ex>[r] \ar@<-1.5ex>[r]
           & \;\;.\,.\,. \;\;\;  \ar[r] 
           & \;\;\; \cc{D}
         }
$$
It follows from the results in \cite{DY} that the site structures in
the categories $\cc{D}_A$ determine a site structure on $\cc{D}$ in
such a way that we have an inverse bi-limit of the topoi of sheaves:
$$
\xymatrix@C=8ex
         {
           & .\,.\,. & \cc{P}_B\;\;   
           & \;  \cc{P}_A\;\; \ar@<1.5ex>[l]^{u}
  \ar@<-1.5ex>[l]_{v}^{\Downarrow \theta_{vu}}
           & \;\;.\,.\,. \;\;\;  \ar@<1.5ex>[l] \ar@<-1.5ex>[l] 
           & \;\;\; \cc{P} \ar[l] 
         }
$$
where the topos $\cc{P}$ is the topos of sheaves on the site $\cc{D}$.

From proposition \ref{2cells} a) it follows that all the topoi
$\cc{P}_A$ are pointed, and from b) it follows that in this case these
points induce a point $\cc{S} \to \cc{P}$ of the bi-limit topos $\cc{P}$. Finally, from
propositions \ref{2cells} d) and \ref{equivalencia}, it follows that
the topoi $\cc{E}$ and $\cc{P}$ are equivalent topoi, $\cc{E}
\mr{\sim} \cc{P}$. Thus, $\cc{E}$ has a point $\cc{S} \to \cc{E}$ determined by any
inverse equivalence and the point of $\cc{P}$. 


\end{document}